\documentclass[reqno]{amsart}
\usepackage{graphicx}
\usepackage{hyperref}
\title[Projective distance inequalities]{Some projective distance inequalities for simplices in complex projective space}
\author{Mark Fincher}
\address{Department of Mathematics\\
University of North Texas\\
1155 Union Circle \#311430\\
Denton, TX 76203}
\email{mfincher777@gmail.com}
\thanks{Financial support provided to Mark Fincher in the form of a UNT SUMS fellowship as part of the UNT mathematics department's NSF funded RTG grant DMS-0943870}
\author{Heather Olney}
\address{Department of Mathematics\\
University of North Texas\\
1155 Union Circle \#311430\\
Denton, TX 76203}
\email{HeatherOlney@my.unt.edu}
\thanks{Financial support provided to Heather Olney in the form of a UNT SUMS fellowship as part of the
UNT mathematics department's NSF funded RTG grant DMS-0943870}
\author{William Cherry}
\address{Department of Mathematics\\
University of North Texas\\
1155 Union Circle \#311430\\
Denton, TX 76203}
\email{wcherry@unt.edu}
\newtheorem{theorem}{Theorem}
\newtheorem{proposition}[theorem]{Proposition}
\newtheorem{cor}[theorem]{Corollary}
\newtheorem{conjecture}[theorem]{Conjecture}
\theoremstyle{remark}
\newtheorem*{remark*}{Remark}
\newtheorem{example}[theorem]{Example}
\date{August 19, 2014}
\begin{document}
\begin{abstract}We prove inequalities relating the absolute value of the
determinant of $n+1$ linearly independent unit vectors in $\mathbf{C}^{n+1}$
and the projective distances from the vertices to the hyperplanes
containing the opposite faces of the simplices in complex projective $n$-space
whose vertices or faces are determined by the given vectors.
\end{abstract}
\maketitle
A basis of unit vectors in $\mathbf{C}^{n+1}$ determines the vertices (or the
faces) of a simplex in $n$-dimensional complex projective space. For reasons
originally motivated by an inequality in complex function theory proven
by Eremenko and the third author \cite{CherryEremenko}, we investigated the
relationship between the determinant of the vectors forming the basis and
the projective distances from each vertex of the simplex to the hyperplane 
containing the face of the opposite side. We show that if $d_{\min}$ denotes
the minimum of these projective distances and if $D$ denotes the determinant of
the basis vectors, then $d_{\min}^n\le|D|\le d_{\min}.$
\par\smallskip
\textit{Acknowledgments.} Surya Raghavendran, during
a research experiences for undergraduates project supervised by the
third author and funded by a SUMS fellowship as part of the
UNT Mathematics Department's NSF funded RTG grant in the summer of 2012, 
made initial investigations into the relationship between the singular values
of the matrix formed by three unit vectors in $\mathbf{C}^3$ and the
projective side lengths of the corresponding projective triangle in
$\mathbf{CP}^2.$ Our results here build upon his initial work.
We also thank Charles Conley for a stimulating discussion that led us to
the proof of Proposition~\ref{lagrange}.
\par\smallskip
Let $\mathbf{e}_0,\dots\mathbf{e}_n$ be a basis for $\mathbf{C}^{n+1}.$
Given two vectors $\mathbf{a}=a_0\mathbf{e}_0+\dots+a_n\mathbf{e}_n$
and $\mathbf{b}=b_0\mathbf{e}_0+\dots+b_n\mathbf{e}_n$ in $\mathbf{C}^{n+1},$
we use $\mathbf{a}\cdot\mathbf{b}$ to denote the standard dot product,
$$
\mathbf{a}\cdot\mathbf{b}=a_0b_0+\dots+a_nb_n,
$$
rather than the Hermitian inner-product more typically used
with complex vector spaces. Thus, in our notation,
$$
|\mathbf{a}|^2=\mathbf{a}\cdot\overline{\mathbf{a}},
$$
where the bar denotes complex conjugation, as usual.
\par\smallskip
For $k=0,\dots,n+1,$ we let $\Lambda^k\mathbf{C}^{n+1}$ denote the $k$-th
exterior power of the vector space $\mathbf{C}^{n+1},$ and we recall that
$$
\mathbf{e}_0\wedge\mathbf{e}_1\wedge\dots\mathbf{e}_{k-1},\dots,
\mathbf{e}_{i_1}\wedge\mathbf{e}_{i_2}\wedge\dots\wedge\mathbf{e}_{i_k},\dots,
\mathbf{e}_{n+1-k}\wedge\mathbf{e}_{n+2-k}\wedge\dots\wedge\mathbf{e}_n,
$$
where $0\le i_1 < i_2 < \dots < i_k \le n$
form a basis for $\Lambda^k\mathbf{C}^{n+1}.$ By declaring this basis to
be orthonormal in $\Lambda^k\mathbf{C}^{n+1},$ the norm and dot product on
$\mathbf{C}^{n+1}$ extends to a norm and inner product on
$\Lambda^k\mathbf{C}^{n+1}.$ For a detailed introduction to exterior
algebras and wedge products, see \cite{BowenWang}.
\begin{proposition}\label{dotofwedge}
Let $1\le k \le n+1$ be an integer, and let $\mathbf{v}_1,\dots,\mathbf{v}_k$
and $\mathbf{w}_1,\dots,\mathbf{w}_k$ be vectors in $\mathbf{C}^{n+1}.$
Then,
$$
(\mathbf{v}_1\wedge\dots\wedge\mathbf{v}_k)\cdot
(\mathbf{w}_1\wedge\dots\wedge\mathbf{w}_k) =
\det(\mathbf{v}_i\cdot\mathbf{w}_j)_{1\le i,j \le k}.
$$
\end{proposition}

\begin{remark*}The matrix of dot products on the right is called a
\textit{Gramian} matrix.
\end{remark*}
\begin{proof}
This is Exercise~39.3 in \cite{BowenWang}.
\end{proof}

\begin{cor}\label{normofwedge}
Let $\mathbf{v}_1,\dots,\mathbf{v}_k$ be $k$ vectors in $\mathbf{C}^{n+1}.$
Then,
$$
|\mathbf{v}_1\wedge\dots\wedge\mathbf{v}_k|^2 = 
\det(\mathbf{v}_i\cdot\overline{\mathbf{v}}_j)_{1\le i,j \le k}.
$$
\end{cor}

\begin{cor}\label{normofwedgeineq}
Let $\mathbf{v}_1,\dots,\mathbf{v}_k$ be $k$ vectors in $\mathbf{C}^{n+1}.$
Then,
$$
|\mathbf{v}_1\wedge\dots\wedge\mathbf{v}_k| \le |\mathbf{v}_1|\cdot\dots\cdot
|\mathbf{v}_k|.
$$
Equality holds if and only if one of the vectors is the zero vector
or if $\mathbf{v}_i\cdot\overline{\mathbf{v}}_j=0$
for all $i\ne j.$
\end{cor}

\begin{proof}
If any of the vectors $\mathbf{v}_j$ are the zero vector, then the inequality
is obvious. So, assume that none of the $\mathbf{v}_j$ are zero. Let
$$
\mathbf{u}_j = \frac{\mathbf{v}_j}{|\mathbf{v}_j|}
$$
be unit vectors in the directions of the $\mathbf{v}_j.$ Then, clearly,
$$
|\mathbf{v}_1\wedge\dots\wedge\mathbf{v}_k|=
\Big||\mathbf{v}_1|\mathbf{u}_1\wedge\dots\wedge
|\mathbf{v}_k|\mathbf{u}_k\Big|=|\mathbf{v}_1|\cdot\dots\cdot|\mathbf{v}_k|
\cdot|\mathbf{u}_1\wedge\dots\wedge\mathbf{u}_k|.
$$
Thus, it suffices to show that
$|\mathbf{u}_1\wedge\dots\wedge\mathbf{u}_k| \le 1.$
To this end, by Corollary~\ref{normofwedge},
\begin{equation}\label{normofwedgeeqn}
|\mathbf{u}_1\wedge\dots\wedge\mathbf{u}_k|^2 =
\det (\mathbf{u}_i\cdot\overline{\mathbf{u}}_j).
\end{equation}
The matrix $(\mathbf{u}_i\cdot\overline{\mathbf{u}}_j)$ is
a $k\times k$ Hermitian matrix, and hence has non-negative eigenvalues
$\lambda_1,\dots,\lambda_k.$ Thus, by the geometric-arithmetic mean
inequality
$$
\det (\mathbf{u}_i\cdot\overline{\mathbf{u}}_j)=
\lambda_1\cdot\dots\cdot\lambda_k \le
\left[\frac{\lambda_1+\dots+\lambda_k}{k}\right]^k = 1,
$$
where the equality on the right follows from the fact that
$$
\lambda_1+\dots+\lambda_k=\mathrm{Trace}
(\mathbf{u}_i\cdot\overline{\mathbf{u}}_j)=k,
$$
since $\mathbf{u}_i\cdot\overline{\mathbf{u}}_i=1.$
\par\smallskip
Equality holds in the arithmetic-geometric mean inequality if and only if
all the eigenvalues are equal, and hence all equal to one. This is the 
case if and only if $(\mathbf{u}_i\cdot\overline{\mathbf{u}}_j)$ is
the $k\times k$ identity matrix, which happens if and only if
$\mathbf{v}_i\cdot\overline{\mathbf{v}}_j=0$ for all $i\ne j.$
\end{proof}

We will be most interested in the $n$-th exterior power of
$\mathbf{C}^{n+1},$ where 
$$
\mathbf{e}_1\wedge\dots\wedge\mathbf{e}_n, \quad\dots,\quad
\mathbf{e}_0\wedge\dots\wedge\mathbf{e}_{j-1}\wedge\mathbf{e}_{j+1}\wedge\dots
\wedge\mathbf{e}_n,\quad\dots\quad,
\mathbf{e}_0\wedge\dots\wedge\mathbf{e}_{n-1}
$$
form a basis of $\Lambda^{n}\mathbf{C}^{n+1}.$
Let $L$ denote the  isometric isomorphism  from
$\Lambda^{n}\mathbf{C}^{n+1}$ to $\mathbf{C}^{n+1}$
defined on the basis vectors as follows:
\begin{align*}
L(\mathbf{e}_1\wedge\dots\wedge\mathbf{e}_n)&=\mathbf{e}_0,\\
&\vdots\\
L(\mathbf{e}_0\wedge\dots\wedge\mathbf{e}_{j-1}\wedge\mathbf{e}_{j+1}\wedge\dots
\wedge\mathbf{e}_n)&=(-1)^j\mathbf{e}_j,\\
&\vdots\\
L(\mathbf{e}_0\wedge\dots\wedge\mathbf{e}_{n-1})&=(-1)^n\mathbf{e}_n.
\end{align*}
Observe that if $n=2$ and $\mathbf{a}$ and $\mathbf{b}$ are vectors in
$\mathbf{C}^3,$ then
$L(\mathbf{a}\wedge\mathbf{b})=\mathbf{a}\times\mathbf{b},$ where
the product on the right is the ordinary cross product in $\mathbf{C}^3.$
\par\smallskip
We will use $L(\mathbf{b}_1,\dots,\mathbf{b}_n)$ as a generalized
cross-product.
\begin{proposition}\label{boxproduct}
Let $\mathbf{a},$ $\mathbf{b}_1, \dots, \mathbf{b}_n$ be $n+1$
vectors in $\mathbf{C}^{n+1}.$ Then,
$$
\det(\mathbf{a},\mathbf{b}_1,\dots,\mathbf{b}_n)=\mathbf{a}\cdot
L(\mathbf{b}_1\wedge\dots\wedge\mathbf{b}_n).
$$
\end{proposition}
\begin{proof}
If we compute the determinant of the $(n+1)\times(n+1)$ matrix whose rows
are $\mathbf{a},$ $\mathbf{b}_1,\dots,\mathbf{b}_n,$ then the expression
on the right is nothing other than the computation of the determinant by
expansion of minors along the first row.
\end{proof}
\begin{cor}The vector $L(\mathbf{b}_1,\dots,\mathbf{b}_n)$ is orthogonal
to each of the $\mathbf{b}_j.$
\end{cor}

We define an equivalence relation on $\mathbf{C}^{n+1}\setminus\{0\}$
by declaring that two non-zero vectors $\mathbf{v}$ and $\mathbf{w}$ in 
$\mathbf{C}^{n+1}$ are equivalent if there exists a non-zero complex scalar
$c$ such that $\mathbf{v}=c\mathbf{w}.$ The set of all such equivalence
classes is denoted by $\mathbf{CP}^n$ and is called the 
\textit{complex projective space} of dimension $n.$ A point in $\mathbf{CP}^n$ is an equivalence class
of vectors in $\mathbf{C}^{n+1}$ and by the definition of the equivalence
relation, we can always represent a point in $\mathbf{CP}^n$ by a unit
vector in $\mathbf{C}^{n+1}.$ The equivalence classes associated with
the vectors in a $k+1$ dimensional subspace of $\mathbf{C}^{n+1}$ is
a $k$-dimensional subspace of $\mathbf{CP}^n.$ When $k=n-1,$ such a subspace
is called a hyperplane in $\mathbf{CP}^n.$ We say that $n+1$ points in
$\mathbf{CP}^n$ are in \textit{general position} if they are not all
contained in any one hyperplane. This is equivalent to the vectors representing
the points being linearly independent in $\mathbf{C}^{n+1}.$
Similarly, we say that $n+1$ hyperplanes in $\mathbf{CP}^n$ are in
\textit{general position} if there is no point in $\mathbf{CP}^n$ contained
in all the hyperplanes. Note that a non-zero vector $\mathbf{v}$ in
$\mathbf{C}^{n+1}$ can be thought of as representing a hyperplane where the
points in the hyperplane are represented by the vectors $\mathbf{x}$ in
$\mathbf{C}^{n+1}$ such that $\mathbf{v}\cdot\mathbf{x}=0.$
\par\smallskip
If $\mathbf{v}$ and $\mathbf{w}$ are two unit vectors in 
$\mathbf{C}^{n+1}$ representing points in $\mathbf{CP}^n,$ then the
\textit{Fubini-Study distance} between the two points is defined to be
$|\mathbf{v}\wedge\mathbf{w}|.$ Now let $\mathbf{u}$ and $\mathbf{v}$
be unit vectors in $\mathbf{C}^{n+1}.$ We think of $\mathbf{u}$
as representing a point in $\mathbf{CP}^n$ and $\mathbf{v}$ as representing
a hyperplane in $\mathbf{CP}^n.$ Then, The Fubini-Study distance from
the point represented by $\mathbf{u}$ to the hyperplane represented
by $\mathbf{v}$ is defined by
\begin{align*}
\textnormal{distance from the point~}\mathbf{u}
&\textnormal{~to the hyperplane~}\mathbf{v}\\
&=\min\{\textnormal{distance from~}\mathbf{u}\textnormal{~to~}\mathbf{x} :
\mathbf{v}\cdot\mathbf{x}=0\textnormal{~and~}|\mathbf{x}|=1\}\\
&=\min\{|\mathbf{u}\wedge\mathbf{x}| : \mathbf{v}\cdot\mathbf{x}=0
\textnormal{~and~}|\mathbf{x}|=1\}.
\end{align*}
\par\smallskip
Second perhaps only to hyperbolic geometry, projective geometry, which arose
out of the study of perspective in classical painting, is among the 
most ubiquitous of the non-Euclidean geometries encountered in 
modern mathematics. See, for instance \cite{RichterGebert} for a recent
accessible introduction.
\par\smallskip
Our first result is a convenient formula for the distance from a vertex of
a projective simplex to the hyperplane determined by the opposite face in
the simplex.

\begin{proposition}\label{disttoopposite}
Let $\mathbf{a},$ $\mathbf{b}_1, \dots, \mathbf{b}_n$ be $n+1$ linearly independent unit vectors in $\mathbf{C}^{n+1}$ representing $n+1$ points in general position in $\mathbf{CP}^n.$ Then, the Fubini-Study distance $d$ from the point $\mathbf{a}$ to the hyperplane in $\mathbf{CP}^n$ spanned by $\mathbf{b}_1,\dots,\mathbf{b}_n$ is given by
$$
d=\frac{|\det(\mathbf{a},\mathbf{b}_1,\dots,\mathbf{b}_n)|}{|\mathbf{b}_1\wedge\dots\wedge\mathbf{b}_n|}.
$$
\end{proposition}
\begin{proof}
Without loss of generality, by making an orthogonal change of coordinates,
we may choose our standard basis vectors 
$\mathbf{e}_0, \dots, \mathbf{e}_n$ in $\mathbf{C}^{n+1}$ so that
$\mathbf{e}_0\cdot \mathbf{b}_j=0$ for $j=1,\dots,n.$ Let $\mathbf{u}$
be a unit vector in the span of $\{\mathbf{b}_1,\dots,\mathbf{b}_n\}.$
Then,
$$
\mathbf{u}=u_1\mathbf{e}_1+\dots+u_n\mathbf{e}_n \qquad\textnormal{with}\qquad
|u_1|^2+\dots+|u_n|^2=1.
$$
Let $\mathbf{a}=a_0\mathbf{e}_0+\dots+a_n\mathbf{e}_n.$
Then, the Fubini-Study distance from the point in $\mathbf{CP}^n$ represented
by $\mathbf{a}$ to the point in $\mathbf{CP}^n$ represented by $\mathbf{u}$
is given by $|\mathbf{a}\wedge\mathbf{u}|.$ Note that
\begin{equation}\label{awedgeueq}
\mathbf{a}\wedge\mathbf{u} = a_0u_1 \mathbf{e}_0\wedge\mathbf{e}_1
+ \dots + a_0u_n \mathbf{e}_0\wedge\mathbf{e}_n+\sum_{1\le i < j \le n}
(a_iu_j-a_ju_i)\mathbf{e}_i\wedge\mathbf{e}_j.
\end{equation}
Hence,
\begin{equation}\label{awedgeuge}
|\mathbf{a}\wedge\mathbf{u}|^2\ge|a_0u_1|^2+\dots+|a_0u_n|^2=|a_0|^2(|u_1|^2+\dots+|u_n|^2)=|a_0|^2.
\end{equation}
Now,
$$
\det(\mathbf{a},\mathbf{b}_1,\dots,\mathbf{b}_n)=\mathbf{a}\cdot
L(\mathbf{b}_1\wedge\dots\wedge\mathbf{b}_n)
$$
by Proposition~\ref{boxproduct}.
Of course, $L(\mathbf{b}_1\wedge\dots\wedge\mathbf{b}_n)$ is orthogonal
to each of the $\mathbf{b}_j.$  By our choice of basis, $\mathbf{e}_0$ is 
also orthogonal to each of the $\mathbf{b}_j.$ Since the $\mathbf{b}_j$
form a set of $n$ linearly independent vectors in an $n+1$-dimensional
vector space, there is only one direction simultaneously orthogonal to
all of the $\mathbf{b}_j.$ Thus,
$L(\mathbf{b}_1\wedge\dots\wedge\mathbf{b}_n)$
is in the span of $\mathbf{e}_0,$ and so
$$
|\mathbf{a}\cdot L(\mathbf{b}_1\wedge\dots\wedge\mathbf{b}_n)|
= |a_0|\cdot|L(\mathbf{b}_1\wedge\dots\wedge\mathbf{b}_n)|.
$$
Thus, observing that 
$$
|L(\mathbf{b}_1\wedge\dots\wedge\mathbf{b}_n)|=
|\mathbf{b}_1\wedge\dots\wedge\mathbf{b}_n|,
$$
we see from $(\ref{awedgeuge})$ that
\begin{align*}
|\mathbf{a}\wedge\mathbf{u}|\ge|a_0|&=\frac{|a_0|\cdot|L(\mathbf{b}_1\wedge\dots\wedge\mathbf{b}_n)|}{|\mathbf{b}_1\wedge\dots\wedge\mathbf{b}_n|}\\
&= \frac{|\mathbf{a}\cdot L(\mathbf{b}_1\wedge\dots\wedge\mathbf{b}_n)|}{|\mathbf{b}_1\wedge\dots\wedge\mathbf{b}_n|}\\
&=\frac{|\det(\mathbf{a},\mathbf{b}_1,\dots,\mathbf{b}_n)|}{|\mathbf{b}_1\wedge\dots\wedge\mathbf{b}_n|}.
\end{align*}
\par
To complete the proof, we need to show that equality
is obtained for some choice of $\mathbf{u}.$ There are two cases.
If $\mathbf{a}$ is the direction of $\mathbf{e}_0,$ then equality holds for
any choice of $\mathbf{u}$ since $a_1=\dots=a_n=0.$ Otherwise, if we choose
$$
u_j = \frac{a_j}{\sqrt{|a_1|^2+\dots+|a_n|^2}},\qquad\textnormal{for~}j=1,\dots,
n,
$$
we see that the terms in the sum on the far right of $(\ref{awedgeueq})$
are all zero, and so equality holds in $(\ref{awedgeuge}).$
\end{proof}
\begin{cor}\label{distgedet}
Let $\mathbf{a},$ $\mathbf{b}_1, \dots, \mathbf{b}_n$ and $d$
be as in Proposition~\ref{disttoopposite}. Then,
$$
d\ge\det(\mathbf{a},\mathbf{b}_1,\dots,\mathbf{b}_n).
$$
Equality holds if and only if $\mathbf{b}_i\cdot\overline{\mathbf{b}}_j=0$
for all $i\ne j.$
\end{cor}
\begin{example}\label{isosceles}
When $n=3,$ let $0<s\le1$ and consider the projective triangle with vertices represented
by the unit vectors
$$
\mathbf{a}=\left[\sqrt{\frac{1-s^2}{2}},\sqrt{\frac{1-s^2}{2}},s\right],\qquad
\mathbf{b}_1=[1,0,0],\qquad\textnormal{and}\qquad
\mathbf{b}_2=[0,1,0].
$$
Then, $|\mathbf{b}_1\wedge\mathbf{b}_2|=1,$ and so
$d=\det(\mathbf{a},\mathbf{b}_1,\mathbf{b_2})=s,$ and equality holds
in Corollary~\ref{distgedet}. We remark that geometrically, these triangles
are isosceles with projective side lengths:
$$
1, \sqrt{\frac{1+s^2}{2}}, \sqrt{\frac{1+s^2}{2}}.
$$
\end{example}
\begin{proof}[Proof of Corollary~\ref{distgedet}.]
By Corollary~\ref{normofwedgeineq},
$|\mathbf{b}_1\wedge\dots\wedge\mathbf{b}_n|\le 1.$ Hence, by the formula
for $d$ in Proposition~\ref{disttoopposite},
$$
d=\frac{\det(\mathbf{a},\mathbf{b}_1,\dots,\mathbf{b}_n)}
{|\mathbf{b}_1\wedge\dots\wedge\mathbf{b}_n|}\ge
\det(\mathbf{a},\mathbf{b}_1,\dots,\mathbf{b}_n).
$$
Equality holds if and only if equality holds in
Corollary~\ref{normofwedgeineq}.
\end{proof}

\begin{proposition}\label{lagrange}
Let $\mathbf{v}_1,\dots,\mathbf{v}_{n-1}$ be $n-1$ linearly independent vectors
in $\mathbf{C}^{n+1}$ and let $\mathbf{w}_1,\dots,\mathbf{w}_n$ be $n$ linearly
independent vectors in $\mathbf{C}^{n+1}.$ If we let
$$
\mathbf{a}=L(\mathbf{w}_1\wedge\dots\wedge\mathbf{w}_n)\qquad
\textnormal{and}\qquad
\mathbf{b}=L(\mathbf{v}_1\wedge\dots\wedge\mathbf{v}_{n-1}\wedge\mathbf{a})
$$
Then
\begin{equation*}
\mathbf{b}= (-1)^n
\det\left(\begin{array}{ccc}
\mathbf{w}_1&\dots&\mathbf{w}_n\\
\mathbf{v}_1\cdot\mathbf{w}_1&\dots&
\mathbf{v}_1\cdot\mathbf{w}_n\\
\vdots&\vdots&\vdots\\
\mathbf{v}_{n-1}\cdot\mathbf{w}_1&\dots&
\mathbf{v}_{n-1}\cdot\mathbf{w}_n
\end{array}
\right).
\end{equation*}
\end{proposition}
\begin{remark*}Note that the matrix specified in the proposition has
vector entries in its first row, and hence its ``determinant'' 
results in a vector.
This Proposition is a generalization of Lagrange's formula for the
vector triple product in $\mathbf{R}^3.$ The proof of this proposition
was inspired by a discussion the last author had with Charles Conley,
and we thank him for his interest.
\end{remark*}
\begin{remark*}We suspect that Proposition~\ref{lagrange}
is probably reasonably well-known, but we were unable to find
a reference to it in the literature.
\end{remark*}
\begin{proof}
Let 
\begin{equation*}
\widetilde{\mathbf{b}}= 
\det\left(\begin{array}{ccc}
\mathbf{w}_1&\dots&\mathbf{w}_n\\
\mathbf{v}_1\cdot\mathbf{w}_1&\dots&
\mathbf{v}_1\cdot\mathbf{w}_n\\
\vdots&\vdots&\vdots\\
\mathbf{v}_{n-1}\cdot\mathbf{w}_1&\dots&
\mathbf{v}_{n-1}\cdot\mathbf{w}_n
\end{array}
\right).
\end{equation*}
We want to show that $\mathbf{b}=(-1)^n\widetilde{\mathbf{b}},$
and for this, it suffices to show that for all
$\mathbf{z}$ in $\mathbf{C}^{n+1},$ we have
$\mathbf{z}\cdot\mathbf{b}=(-1)^n\mathbf{z}\cdot\widetilde{\mathbf{b}}.$
Clearly,
$$
\mathbf{z}\cdot\widetilde{\mathbf{b}}=\det\left(\begin{array}{ccc}
\mathbf{z}\cdot\mathbf{w}_1&\dots&\mathbf{z}\cdot\mathbf{w}_n\\
\mathbf{v}_1\cdot\mathbf{w}_1&\dots&\mathbf{v}_1\cdot\mathbf{w}_n\\
\vdots&\vdots&\vdots\\
\mathbf{v}_{n-1}\cdot\mathbf{w}_1&\dots&\mathbf{v}_{n-1}\cdot\mathbf{w}_n
\end{array}\right).
$$
On the other hand, by Proposition~\ref{boxproduct},
\begin{align*}
\mathbf{z}\cdot\mathbf{b}&=\det(\mathbf{z},\mathbf{v}_1,\dots,\mathbf{v}_{n-1},
\mathbf{a})\\
&=(-1)^n\det(\mathbf{a},\mathbf{z},\mathbf{v}_1,\dots,\mathbf{v}_{n-1})\\
&=(-1)^n\mathbf{a}\cdot L(\mathbf{z}\wedge\mathbf{v}_1\wedge\dots\wedge\mathbf{v}_{n-1})\\
&=(-1)^nL(\mathbf{w}_1\wedge\dots\wedge\mathbf{w}_n)\cdot
L(\mathbf{z}\wedge\mathbf{v}_1\wedge\dots\wedge\mathbf{v}_{n-1})\\
&=(-1)^n(\mathbf{w}_1\wedge\dots\wedge\mathbf{w}_n)\cdot
  (\mathbf{z}\wedge\mathbf{v}_1\wedge\dots\wedge\mathbf{v}_{n-1})\\
&\qquad\textnormal{[since}~L~\textnormal{is an isometry]}\\
&=(-1)^n(\mathbf{z}\wedge\mathbf{v}_1\wedge\dots\wedge\mathbf{v}_{n-1})\cdot
(\mathbf{w}_1\wedge\dots\wedge\mathbf{w}_n)\\
&=(-1)^n\det\left(\begin{array}{ccc}
\mathbf{z}\cdot\mathbf{w}_1&\dots&\mathbf{z}\cdot\mathbf{w}_n\\
\mathbf{v}_1\cdot\mathbf{w}_1&\dots&\mathbf{v}_1\cdot\mathbf{w}_n\\
\vdots&\vdots&\vdots\\
\mathbf{v}_{n-1}\cdot\mathbf{w}_1&\dots&\mathbf{v}_{n-1}\cdot\mathbf{w}_n
\end{array}\right)
\end{align*}
by Proposition~\ref{dotofwedge}.
\end{proof}
\begin{proposition}\label{multicross}
Let $\mathbf{a},\mathbf{u}_1,\dots,\mathbf{u}_n$ be $n+1$ linearly
independent vectors in $\mathbf{C}^{n+1}.$ For $j=1,\dots,n,$ let
$$
\mathbf{v}_j=L(\mathbf{a}\wedge\mathbf{u}_1\wedge\dots\wedge
\mathbf{u}_{j-1}\wedge\mathbf{u}_{j+1}\wedge\dots\wedge\mathbf{u}_{n}).
$$
Then,
$L(\mathbf{v}_1\wedge\dots\wedge\mathbf{v}_{n})=\pm D^{n-1}\mathbf{a},$
where $D=\det(\mathbf{a},\mathbf{u}_1,\dots,\mathbf{u}_n).$
\end{proposition}
\begin{remark*}The unspecified sign depends only on $n$ and can be
explicitly determined from the proof. Since the sign will not matter for
our purpose, we did not bother to record it here.
\end{remark*}
\begin{proof}
By Proposition~\ref{lagrange}, we get that
$$
L(\mathbf{v}_1\wedge\dots\wedge\mathbf{v}_n)=(-1)^n\det\left(\begin{array}{cccc}
\mathbf{a}&\mathbf{u}_1&\dots&\mathbf{u}_{n-1}\\
\mathbf{v}_1\cdot\mathbf{a}&\mathbf{v}_1\cdot\mathbf{u}_1&\dots&\mathbf{v}_1\cdot\mathbf{u}_{n-1}\\
\vdots&\vdots&\vdots&\vdots\\
\mathbf{v}_{n-1}\cdot\mathbf{a}&\mathbf{v}_{n-1}\cdot\mathbf{u}_1&\dots&\mathbf{v}_{n-1}\cdot\mathbf{u}_{n-1}
\end{array}\right).
$$
If $i\ne j,$ then
$$
\mathbf{v}_i\cdot\mathbf{u}_j = 
L(\mathbf{a}\wedge\dots\wedge\mathbf{u}_{i-1}\wedge\mathbf{u}_{i+1}\wedge
\dots\wedge\mathbf{u}_n)\cdot\mathbf{u}_j =0,
$$
since $\mathbf{u}_j$ appears in the wedge product defining $\mathbf{v}_i,$
and hence $\mathbf{v}_i$ is orthogonal to $\mathbf{u}_j.$ Similarly,
$\mathbf{v}_i\cdot\mathbf{a}=0.$ Moreover,
$$
\mathbf{v}_j\cdot\mathbf{u}_j = 
L(\mathbf{a}\wedge\dots\wedge\mathbf{u}_{j-1}\wedge\mathbf{u}_{j+1}\wedge
\dots\wedge\mathbf{u}_n)\cdot\mathbf{u}_j=(-1)^jD,
$$
by Proposition~\ref{boxproduct}. Hence,
$$
L(\mathbf{v}_1\wedge\dots\wedge\mathbf{v}_n)
=(-1)^n\det\left(\begin{array}{ccccc}
\mathbf{a}&\mathbf{u}_1&\mathbf{u}_2&\dots&\mathbf{u}_{n-1}\\
0&-D&0&\dots&0\\
0&0&D&\dots&0\\
\vdots&\vdots&\vdots&\vdots&\vdots\\
0&0&0&\dots&(-1)^{n-1}D
\end{array}\right)=\pm D^{n-1}\mathbf{a}.\qedhere
$$
\end{proof}
\begin{theorem}\label{dnledet}
Let $\mathbf{u}_0,\dots,\mathbf{u}_n$ be $n+1$ linearly independent unit
vectors in $\mathbf{C}^{n+1}$ representing $n+1$ points in general position in
$\mathbf{CP}^n,$ which we think of as the vertices of a projective simplex.
For each $j$ from $0$ to $n,$ let $d_j$ denote the
Fubini-Study distance from the point represented by $\mathbf{u}_j$ to the
hyperplane containing the opposite face of the simplex. Let $d_{\min}$
denote the minimum of the $d_j.$ Then,
$$
d_{\min}^n \le |\det(\mathbf{u}_0,\dots,\mathbf{u}_n)|.
$$
For equality to hold, at least $n$ of the $n+1$ projective distances
$d_j$ must equal $d_{\min}.$
\end{theorem}
\begin{proof}
Let $D=\det(\mathbf{u}_0,\dots\mathbf{u}_n).$
Note that $D\ne0$ by the linear independence (general position) hypothesis.
Without loss of generality, assume that $d_{\min}=d_n.$ Then,
$d_{\min}^n \le d_1d_1\cdot d_n,$ and equality holds if and only if
each of these distances are equal. By Proposition~\ref{disttoopposite},
$$
d_j=\frac{|D|}{|\mathbf{u}_0\wedge\dots\wedge\mathbf{u}_{j-1}\wedge
\mathbf{u}_{j+1}\wedge\dots\wedge\mathbf{u}_n|}.
$$
Thus,
$$
d_{\min}^n \le \frac{|D|^n}{\prod_{j=1}^n |\mathbf{u}_0\wedge\dots\wedge
\mathbf{u}_{j-1}\wedge\mathbf{u}_{j+1}\wedge\dots\wedge\mathbf{u}_n|}.
$$
For $j$ from $1$ to $n,$ let
$$
\mathbf{v}_j=L(\mathbf{u}_0\wedge\dots\wedge\mathbf{u}_{j-1}\wedge
\mathbf{u}_{j+1}\wedge\dots\wedge\mathbf{u}_n),
$$
and we now consider $L(\mathbf{v}_1\wedge\dots\wedge\mathbf{v}_n).$
By Proposition~\ref{multicross},
$$
L(\mathbf{v}_1\wedge\dots\wedge\mathbf{v}_n)=\pm D^{n-1}\mathbf{u}_0.
$$
Hence,
$$
|L(\mathbf{v}_1\wedge\dots\wedge\mathbf{v}_n)|=|D|^{n-1}
$$
since $|\mathbf{u}_0|=1.$ We also know that
$$
|L(\mathbf{v}_1\wedge\dots\wedge\mathbf{v}_n)| =
|\mathbf{v}_1\wedge\dots\wedge\mathbf{v}_n|\le |\mathbf{v}_1|\cdot\dots\cdot
|\mathbf{v}_n|
$$
by Corollary~\ref{normofwedgeineq}. Moreover, the inequality is strict
unless $\mathbf{v}_i\cdot\overline{\mathbf{v}}_j=0$ for all $i\ne j.$
Thus,
\begin{align*}
\prod_{j=1}^n|\mathbf{u}_0\wedge\dots\wedge\mathbf{u}_{j-1}\wedge
\mathbf{u}_{j+1}\wedge\mathbf{u}_{n}|
&=\prod_{j=1}^n|L(\mathbf{u}_0\wedge\dots\wedge\mathbf{u}_{j-1}\wedge
\mathbf{u}_{j+1}\wedge\mathbf{u}_{n})|\\
&=\prod_{j=1}^n|\mathbf{v}_j|\\
&\ge|L(\mathbf{v}_1\wedge\dots\wedge\mathbf{v}_n)|=|D|^{n-1}.
\end{align*}
Hence, 
$$
d_{\min}^n \le \frac{|D|^n}{\prod_{j=1}^n |\mathbf{u}_0\wedge\dots\wedge
\mathbf{u}_{j-1}\wedge\mathbf{u}_{j+1}\wedge\dots\wedge\mathbf{u}_n|}
\le \frac{|D|^n}{|D|^{n-1}}=|D|,
$$
as required, with strict inequality unless $d_1=\dots=d_n$
and $\mathbf{v}_i\cdot\overline{\mathbf{v}}_j=0$ for all $i\ne j.$
\end{proof}
\begin{remark*}Equality of the $n$ distances is not sufficient for
equality to hold in Theorem~\ref{dnledet}, but the proof of
Theorem~\ref{dnledet} suggests the following conjecture.
\end{remark*}
\begin{conjecture}With notation as in Theorem~\ref{dnledet}, fix $0<D\le1$
and consider all configurations of
$\mathbf{u}_0,\dots,\mathbf{u}_n$ such that
$D=|\det(\mathbf{u}_0,\dots,\mathbf{u}_n)|.$ Among all such configurations,
the configuration with the largest $d_{\min}$ will be a regular simplex.
\end{conjecture}
\begin{remark*}
When $D<1,$ equality will not hold in Theorem~\ref{dnledet} for the
regular simplex with determinant $D.$
\end{remark*}
\par
We now observe that if we like, we could just as easily work with vectors
defining the faces of the simplices, rather than the vertices.
\begin{proposition}\label{faces}
Let $\mathbf{a},\mathbf{b}_1,\dots\mathbf{b}_n$ be $n+1$ linearly independent
unit vectors in $\mathbf{C}^{n+1}.$ We think of the vectors as the 
coefficients of linear forms defining hyperplanes in $\mathbf{CP}^n.$
By linear independence, the hyperplanes are in general position and thus
determine a simplex. Let $d$ denote the distance from the hyperplane
determined by $\mathbf{a}$ to the vertex of the simplex where the hyperplanes
determined by $\mathbf{b}_1,\dots\mathbf{b}_n$ intersect. Then,
$$
d=\frac{|\det(\mathbf{a},\mathbf{b}_1,\dots,\mathbf{b_n})|}
{|\mathbf{b}_1\wedge\dots\wedge\mathbf{b}_n|}.
$$
\end{proposition}
\begin{remark*}Observe that the distance formula here is identical
to that in Proposition~\ref{disttoopposite}. Thus, 
Theorem~\ref{dnledet} and Corollary~\ref{distgedet} immediately
translate to the following corollary.
\end{remark*}
\begin{cor}\label{facescor}
Let $\mathbf{u}_0,\dots,\mathbf{u}_n$ be $n+1$ linearly independent unit
vectors in $\mathbf{C}^{n+1}$ representing $n+1$ linear forms
defining $n+1$ hyperplanes in general position in
$\mathbf{CP}^n,$ which we think of as the faces of a projective simplex.
For each $j$ from $0$ to $n,$ let $d_j$ denote the
Fubini-Study distance from the hyperplane represented by $\mathbf{u}_j$ to the
the opposite vertex of the simplex. Let $d_{\min}$
denote the minimum of the $d_j.$ Then,
$$
d_{\min}^n \le |\det(\mathbf{u}_0,\dots,\mathbf{u}_n)| \le d_{\min}.
$$
\end{cor}
\begin{figure}
\begin{picture}(300,300)
\put(10,0){\includegraphics[trim = 0cm 4cm 0cm 5cm, clip, scale=0.5]{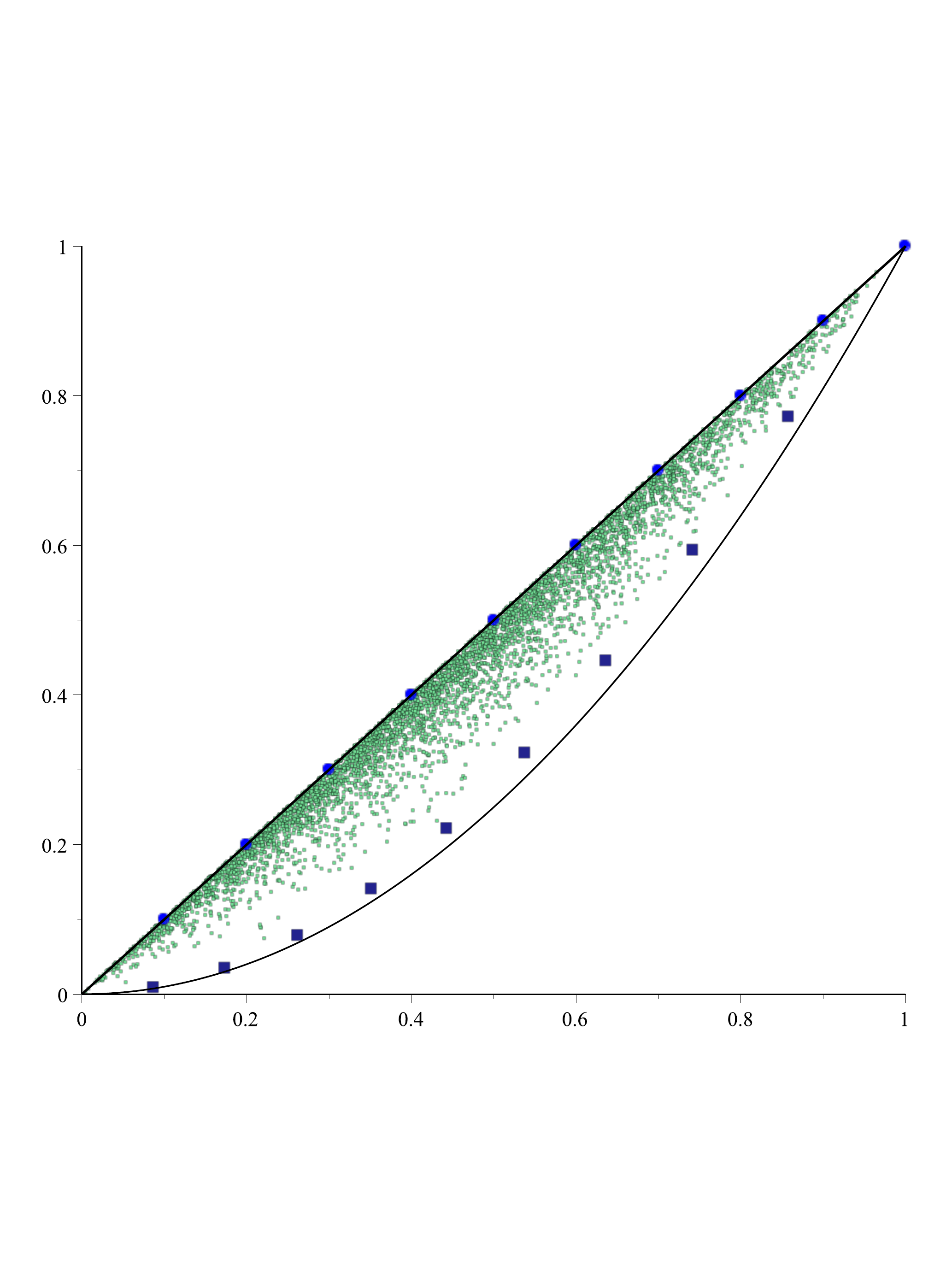}}
\put(190,90){$|D|=d_{\min}^2$}
\put(95,130){$|D|=d_{\min}$}
\put(150,0){$d_{\min}$}
\put(0,150){$|D|$}
\end{picture}
\caption{$|D|$ versus $d_{\min}$ in the case of dimension $n=2.$}\label{fig}
\end{figure}
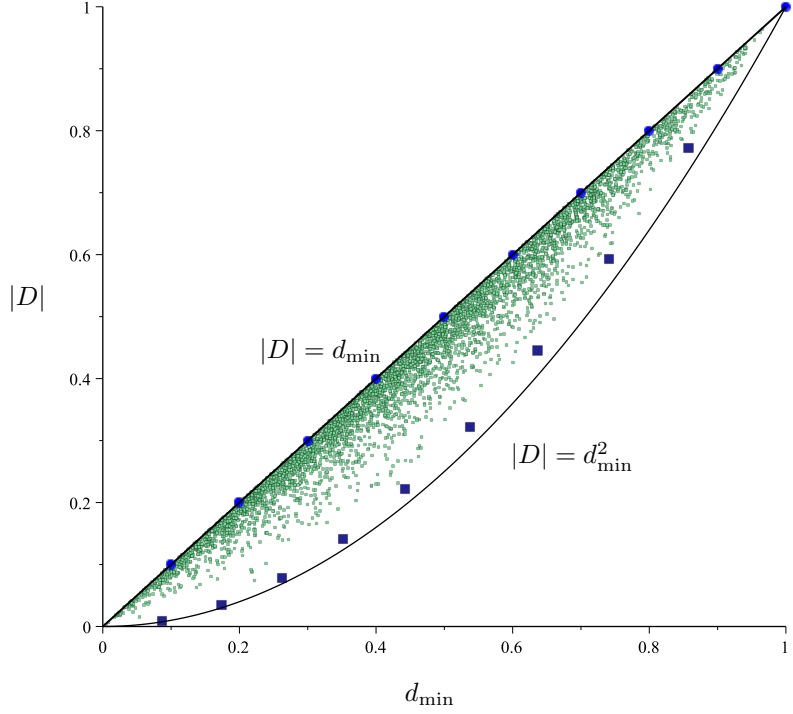
\begin{remark*}Figure~\ref{fig} illustrates the inequalities constraining the
absolute value of the determinant and the minimum distance in the case
when $n=2,$ \textit{i.e.,} for the case of projective triangles in the
projective plane. The points marked as circles along
the line $|D|=d_{\min}$ illustrate isosceles triangles,
as in Example~\ref{isosceles}. The points marked as squares just above
the curve $|D|=d_{\min}^2$ are from equilateral triangles. The other points
are triangles with randomly generated vertices.
\end{remark*}
\begin{proof}[Proof of Proposition~\ref{faces}.]
Let
$$
\mathbf{u}=\frac{L(\mathbf{b}_1\wedge\dots\wedge\mathbf{b}_n)}
{|\mathbf{b}_1\wedge\dots\wedge\mathbf{b}_n|},
$$
which is the unit vector representing the vertex of the simplex where
the hyperplanes determined by $\mathbf{b}_1,\dots,\mathbf{b}_n$
intersect. For $j=1,\dots,n,$ let
$$
\mathbf{v_j}=L(\mathbf{a}\wedge\mathbf{b}_1\wedge\dots\wedge
\mathbf{b}_{j-1}\wedge\mathbf{b}_{j+1}\wedge\dots\wedge
\mathbf{b}_n).
$$
Then, the vectors $\mathbf{v}_j,$ which are not necessarily unit vectors,
represent the $n$ other vertices of the simplex. By
Proposition~\ref{disttoopposite} and Proposition~\ref{boxproduct},
$$
d = \frac{\left|\det\left(\mathbf{u},\frac{\mathbf{v}_1}{|\mathbf{v}_1|},
\dots,\frac{\mathbf{v}_n}{|\mathbf{v}_n|}\right)\right|}
{\left|\frac{\mathbf{v}_1}{|\mathbf{v}_1|}\wedge\dots\wedge
\frac{\mathbf{v}_n}{|\mathbf{v}_n|}\right|}=
\frac{|\mathbf{u}\cdot L(\mathbf{v}_1\wedge\dots\wedge\mathbf{v}_n)|}
{|\mathbf{v}_1\wedge\dots\wedge\mathbf{v}_n|}
$$
By Proposition~\ref{multicross},
$L(\mathbf{v}_1\wedge\dots\wedge\mathbf{v}_n)=\pm D^{n-1}\mathbf{a},$
where $D=\det(\mathbf{a},\mathbf{b}_1,\dots,\mathbf{b}_n).$
Thus,
\begin{align*}
d&=\frac{|u\cdot L(\mathbf{v}_1\wedge\dots\wedge\mathbf{v}_n)|}{
|\mathbf{v}_1\wedge\dots\wedge\mathbf{v}_n|}\\
&=\frac{|D|^{n-1}|\mathbf{u}\cdot\mathbf{a}|}{|D|^{n-1}}
\qquad\textnormal{[since}~\mathbf{a}~\textnormal{is a unit vector]}\\
&=\frac{|L(\mathbf{b}_1\wedge\dots\wedge\mathbf{b}_n)\cdot \mathbf{a}|}
{|\mathbf{b}_1\wedge\dots\wedge\mathbf{b}_n|}
\qquad\textnormal{[by the definition of}~\mathbf{u}]\\
&=\frac{|\det(\mathbf{a},\mathbf{b}_1,\dots,\mathbf{b}_n)|}
{|\mathbf{b}_1\wedge\dots\wedge\mathbf{b}_n|}
\end{align*}
by Proposition~\ref{boxproduct}.
\end{proof}
\par\smallskip
We conclude by explaining some of the initial motivation coming from
complex function theory for this investigation.
Let $\mathbf{D}$ denote the unit disc in the 
complex plane. In the 1940's, J.~Dufresnoy \cite{Dufresnoy}
studied complex analytic mappings $f$ from $\mathbf{D}$ to
$\mathbf{CP}^n$ such that the image of $f$ omits at least $2n+1$
hyperplanes in general position in $\mathbf{CP}^n,$ where here
\textit{general position} means that the linear
forms defining any $n+1$ of the hyperplanes will be linearly independent.
As in \cite{CherryEremenko},
we let $f^\#$ denote the Fubini-Study derivative of $f,$ which measures how
much the mapping $f$ distorts length, where length in $\mathbf{D}$ is measured
with respect to the standard Euclidean metric and length in $\mathbf{CP}^n$
is measured with respect to the Fubini-Study metric. A consequence of
Dufresnoy's work is that $f^\#(0)$ is bounded above by a constant depending 
only on the dimension $n$ and the set of omitted hyperplanes, but 
Defresnoy remarked in his 1944 paper that the constant depends on the
omitted hyperplanes in a ``completely unknown'' way.
By making a portion, \textit{cf.}\ \cite{Eremenko}, of
the potential-theoretic method of Eremenko and Sodin \cite{EremenkoSodin}
effective, Cherry and Eremenko \cite{CherryEremenko}
were able to give an explicit and effective
estimate on how the constant depends on the omitted hyperplanes.
Cherry and Eremenko's bound was expressed in terms of the
singular values
of the $(n+1)\times(n+1)$ matrices formed by the coefficients of 
the normalized linear forms defining $n+1$
of the omitted hyperplanes. Let $P$ be a point
in $\mathbf{CP}^n$ where $n$ of the $2n+1$ omitted hyperplanes intersect,
and let $Q$ be a point where a different $n$ of the $2n+1$ omitted hyperplanes
intersect. Then, the projective line connecting $P$ with $Q$ will intersect
the $2n+1$
omitted hyperplanes in only three points: it will intersect $n$ of the
hyperplanes at $P,$ another $n$ at $Q$ and the last one at some third point
$R.$ Such a line is called a \textit{diagonal} line for the hyperplane
configuration. In the event that the hyperplane configuration is such that
for some diagonal line, two of the three points $P,$ $Q,$ and $R$ are very
close together, it is not hard to see that one can find a complex analytic
map $f$ from $\mathbf{D}$ into the diagonal line omitting the three points
such that $f^\#(0)$ is very large. One is then led to ask if this is the only
way one can get a very large value of $f^\#(0).$
One would thus like to know how this minimum distance among the pairs
of points in $\{P,Q,R\}$ compares to the singular values
appearing in the Cherry-Eremenko bound. Rather than look initially at 
collections of $2n+1$ hyperplanes in $\mathbf{CP}^n,$ we began with the
easier situation of $n+1$ hyperplanes in $\mathbf{CP}^n$ and did some numerical
experiments comparing the singular values of the matrices formed by the
coefficients of the defining forms of the hyperplanes and the projective
distances from the hyperplanes to the opposite vertices of the simplex
whose faces are
contained in the given hyperplanes. These opposite
vertices would be the points determining the diagonal lines in bigger
configurations of hyperplanes. Although the Cherry-Eremenko bound is expressed
only in terms of some of the singular values, we realized that we could
obtain prettier results for the determinant, whose absolute value is of course
the square root of the product of all the singular values.
We therefore decided to write this note
focusing on the pure projective geometry of the simplices and leave the
possible application to complex function theory to another time.

\end{document}